\documentclass[12pt]{article}

\usepackage{a4wide,color}
\usepackage{amsmath,amsfonts,amssymb,amsthm,mathtools}
\usepackage{graphicx}
\usepackage{mathabx}
\usepackage{soul}
\newtheorem{theorem}{Theorem}[section]
\newtheorem{remark}[theorem]{Remark}

\newtheorem{proposition}[theorem]{Proposition}

\newtheorem{lemma}[theorem]{Lemma}

\newcommand{\R}{{\mathbb R}}

\usepackage{comment}

\begin{document}

\title{\Large The asymptotic distribution of the scaled remainder for pseudo golden ratio expansions of a continuous random variable}

\author{I. Herbst$^1$, J. M{\o}ller$^2$, A.M. Svane$^2$}
\date{%
    $1$ Department of Mathematics, University of Virginia, Charlottesville, Va, 22903, U.S.A.\\%
    $2$ Department of Mathematical Sciences, Aalborg University, Aalborg \O, Denmark \\%
    ${}$\\
    \today}

\maketitle
\begin{abstract}
    Let $X=\sum_{k=1}^\infty X_k \beta^{-k}$ be the (greedy) base-$\beta$ expansion of a continuous random variable $X$ on the unit interval where $\beta$ is the positive solution to $\beta^n = 1 + \beta + \cdots + \beta^{n-1}$ for an integer $n\ge 2$ (i.e., $\beta$ is a generalization of the golden mean corresponding to $n=2$).
  We study the asymptotic distribution and convergence rate of the scaled remainder $\sum_{k=1}^\infty X_{m+k} \beta^{-k}$ when $m$ tends to infinity. 
\end{abstract}

\section{Introduction and main results}\label{s:1}

In this paper we consider  an absolutely continuous random variable $X$ with probability density function (PDF) $f$ concentrated on $[0,1)$. For  $\beta>1$ we let $X = \sum_{k=1}^\infty X_k\beta^{-k}$ be its  base $\beta$-expansion with digits $X_k$ defined recursively by $X_1=\lfloor \beta X \rfloor$ and $X_k=\lfloor \beta^k(X- \sum_{i=1}^{k-1} X_i\beta^{-i} ) \rfloor$ for $k\ge2$ where  $\lfloor x \rfloor$ denotes the integer part of $x$. We are interested in the distribution of the scaled remainder $\sum_{k=1}^\infty X_{k+m}\beta^{-k}$ when the first $m$ digits are removed. In the case where $\beta $ is an integer, the limiting distribution of the scaled remainder and the speed at which it converges to this limit has been discussed in \cite {HMSI}. In this paper, we consider the situation where $\beta$ is 
the  positive solution 
to 
\begin{equation}\label{betan}
	\beta^n = 1 + \beta +\cdots + \beta^{n-1}
\end{equation}
for a given integer $n\ge 2$.  The golden mean
corresponds to the case $n=2$ in \eqref{betan}. The study of  $\beta$-expansions when $\beta$ is the golden mean is a classical topic in metric number theory \cite{Fritz}. For $n\ge 2$ the positive solution $\beta=\beta_n$ to \eqref{betan} is called a pseudo golden mean and is an example of a Parry number \cite{Parry1960}.

Let us introduce some notation before stating the main result of this paper.
Given $x\in [0,1)$ and $\beta>1$ let 
$$t(x) = \beta x - \lfloor \beta x \rfloor.$$
Define $t^0(x) = x$ and for $m\in \mathbb{N}$, where $\mathbb{N}$ denotes the natural numbers $1,2,...$, let $t^m$ be the composition of $t$ with itself $m$ times. Define the digits in base $\beta$ as $x_k\coloneqq \lfloor\beta t^{k-1}(x)\rfloor$. 
It is well-known that we have the $\beta$-expansion 
\begin{equation} \label{beta expansion}
	x = \sum_{k=1}^\infty x_k\beta^{-k}
\end{equation} 
and since $\beta$-expansions are not unique this is sometimes referred to as the greedy $\beta$-expansion, see e.g.~\cite{greedy}. Note that 
$0\le t(x)<1$ and $0\le \beta^{-m}t^m(x) = x - \sum_{k=1}^m x_k\beta^{-k} < \beta^{-m}$. 

As shown in Appendix~B, 
there is a unique positive solution to \eqref{betan} and this number is strictly between 1 and 2. Henceforth, let $\beta$ denote this number. It follows that $x_k \in \{0,1\}$ for $k=1,2,...$. 
Not all sequences $(x_1,x_2,...)\in\{0,1\}^{\mathbb N}$ give rise to a $\beta$-expansion of a number in $[0,1)$,  for example $\sum_{k=1}^\infty \beta^{-k} = (\beta-1)^{-1} > 1$. Moreover, due to the relation \eqref{betan}, expansions of the form \eqref{beta expansion} are not necessarily unique.
It follows from Theorem 3, page 406 of \cite{Parry1960} 
that the sequences that arise as $x_k\coloneqq \lfloor\beta t^{k-1}(x)\rfloor$ for some $x\in[0,1)$ are exactly those having no more than $n-1$ $x_k$'s in succession  equal to $1$ and not ending with the sequence $(\ldots,0,1,\ldots,1,,0,1,\ldots,1,\ldots)$ where each sequence of 1's has $n-1$ elements. 
We refer to these conditions as the standard restrictions on the  $x_k$'s.
A self-contained proof  is given in Appendix~A. 

{ 
	The main object of our interest is the scaled remainder 
	\begin{equation}\label{e:main}
		t^m(X) = \beta^m\sum_{k=m+1}^\infty X_k\beta^{-k} = \sum_{k=1}^\infty X_{k+m}\beta^{-k}
	\end{equation} 
	for $m\in \mathbb{N}$. Denote by $\mathcal {B}$ the Borel subsets of  $[0,1)$.  Then $t$ induces a transformation $T$  on PDF's concentrated on $[0,1)$ with $Tf$ characterized by $P(t(X) \in B) = \int_B Tf(x)\,\mathrm dx$ for all $B \in \mathcal {B}$ where $X$ is a random variable having  $f$ as PDF. 
	Then  $ f_m\coloneqq T^m f$ for $m\in \mathbb{N}$ is the PDF of $t^m(X)$, i.e.\ with the definition $P_m (B) \coloneqq P(t^m(X) \in B) $ we have  $P_m(B) =\int_B f_m(x)\,\mathrm dx$. In \cite{Renyi},  R{\'e}nyi showed the existence of a unique PDF $f_\beta$ such that $T f_\beta =f_\beta$.   A formula for  $f_\beta$  (up to a constant factor) for arbitrary 
	$\beta > 1$ is given in \cite{Parry1960}.  
	One may define $f_\beta$  by
	\begin{equation} \label{fbeta1}
		f_\beta(x) =  \left(\sum_{j=1}^n F_j(x)\right)D(\beta)^{-1},\quad 
		F_j = (\beta^j - \beta^{j-1} - \beta^{j-2} - \cdots - \beta)\chi_{[\beta^{-1} + \cdots + \beta^{-(j-1)}, \beta^{-1} + \cdots + \beta^{-j})
		}
	\end{equation}
	for $0\le x<1$ where $\chi_I$ is the indicator function for the interval $I$ and
	$D(\beta)$ is a constant defined so that  $\int_{[0,1)}f_\beta(x)\,\mathrm dx = 1$ (an expression for $D(\beta)$ can be found in \eqref{D} below). In \eqref{fbeta} below we give an alternative formula for $f_\beta$ which is equivalent the one in \cite{Parry1960} (as one may realize by noting that in that paper $t^j(1) = \beta^{-1}+ \cdots + \beta^{n-j}$ for $j\in\mathbb N$). 
	
	Define $P_\beta (B)\coloneqq \int_B f_\beta(x)\,\mathrm dx$ for $ B \in \mathcal{B}$.
	Our main theorem identifies $P_\beta$ as the asymptotic distribution of $t^m(X)$ when $m\to \infty$. To state the theorem, we recall that the total variation distance between two probability distributions $P_1$ and $P_2$ is defined as
	$$d_{TV}(P_1,P_2) = \sup_{B \in \mathcal{B}} |P_1(B) - P_2(B)|$$    Moreover, we use the notation  $\|h-f_\beta\|_\infty = \sup_{x\in |0,1)} |h(x) - f_\beta(x)|$. 
	Finally, we let $\lambda_2$ be the root of the polynomial $p(x) = x^n -( 1 + x + \cdots + x^{n-1})$ with the second largest absolute value among all $n$ roots.} From Appendix~B 
	we know $|\lambda_2| < 1$.

\begin{theorem} \label{mainthm}
	Suppose $X$ is absolutely continuous with PDF $f$ concentrated on $[0,1)$.
	We have
	\begin{equation} \label{dTV}
		\lim_{m\to \infty} d_{TV}(P_m,P_\beta) = 0.   
	\end{equation}
	In addition, if $f$ is a Lipschitz function with $|f(x) - f(y)| \le L_f |x-y|$ for all $x,y \in [0,1)$,
	then for any integer $m>2n$
	\begin{equation} \label{fm-fbeta}
		\|f_m - f_\beta\|_\infty = O(|\beta^{-1}\lambda_2|^{m_2}) + O(\beta^{-m_1}) L_f 
	\end{equation}
	whenever $m_1,m_2 \in \mathbb{N}$ with $m_2 > 2n-1$, $m_1\ge n-1$ and $m_1 +m_2 = m$.
\end{theorem}

\begin{remark}  Here, $m_1$ and $m_2$ should be chosen so that \eqref{fm-fbeta} is optimized once additional information is known about $|\lambda_2|$. This and the values of $\beta$ and $\|\lambda_2|$ as well as the convergence rates of $\|f_m - f_\beta\|_\infty $ in specific cases of $f$ are  investigated in Section~\ref{s:ex}.
	Instead of assuming $f$ is Lipschitz, we could assume that $f$ is only H\"{o}lder continuous with an appropriate change in the exponential rate of convergence in \eqref{fm-fbeta} depending on the H\"{o}lder exponent. 
\end{remark}

We remark that 
\begin{equation}\label{e:bound}
1>|\lambda_2| \ge \beta^{-1/(n-1)}
\end{equation}
 and so $|\lambda_2|$ approaches 1 as $n$ grows.
 For $n = 2$ or $3$ we have the equality $|\lambda_2| = \beta^{-1/(n-1)}$ (see  Appendix~B). 
 Thus under  the Lipschitz continuity assumption in Theorem~\ref{mainthm}, for $n=2$   
$$\|f_m - f_\beta\|_\infty = O(\beta^{-2m/3}),$$ and for $n = 3$
$$\|f_m - f_\beta\|_\infty = O(\beta^{-3m/5}),$$  while for $n>3$ we have no upper bound on $|\lambda_2|$ other than $|\lambda_2| <1$ which yields
\begin{equation}\label{eq:bound_0.5}
\|f_m - f_\beta\|_\infty = O(\beta^{-m/2}).
\end{equation}

The paper is organized as follows. Section~\ref{s:ex} discusses various numerical results in relation to Theorem~\ref{mainthm}.
The remainder of the paper
concerns the proof of Theorem \ref{mainthm}: We first determine the PDF $f_m$ in 
Section~\ref{sec:tm}. Some technical considerations necessary for the proof are given in Section~\ref{sec:Omega} before the main theorem is proved in Section~\ref{sec:proof}. (We mention that a very different proof of the limit \eqref{dTV} will be given in an upcoming paper 
of Cornean, Herbst, and Marcelli by analyzing the operator $T$ defined above.) Finally some technical details are deferred to Appendix~A and B. 

\section{Numerical results}\label{s:ex}

To further examine the convergence rates in Theorem \ref{mainthm}, we computed the value of $\beta$ and $|\lambda_2|$ for $n=2,\ldots,10$ in Table~\ref{tab:beta}. 
It seems that for $n>3$ the lower bound in \eqref{e:bound} 
is strict. Ignoring that $m_1$ and $m_2$ are integers,
the optimal rate in Theorem \ref{mainthm} is achieved when $|\beta^{-1}\lambda_2|^{m_2}=\beta^{-m_1}$, that is,
 $m_1=tm_2$ with $t=\ln(\beta^{-1}|\lambda_2|)/\ln(\beta^{-1})=1-\ln(|\lambda_2|)/\ln(\beta)$. 
Then, since $m_1+m_2=m$, the right hand side in \eqref{fm-fbeta} becomes $O(\beta^{-mt/(1+t)})$. In Table~\ref{tab:beta} we numerically computed $t/(t+1)$ for different values of $n$. We see that it is strictly above the lower bound $0.5$ from \eqref{eq:bound_0.5}, but approaches $0.5$ when $n$ increases.

\begin{table}
\begin{tabular}{cccccccccc}
	\hline
n& 2 & 3 & 4&5&6&7&8&9&10 \\
\hline
$\beta$ & 1.618&1.839& 1.928& 1.966 &1.984& 1.992& 1.996& 1.998&1.999\\
$|\lambda_2|$ & 0.618 & 0.737 &0.818& 0.871& 0.906& 0.930& 0.947& 0.959& 0.968\\
$\beta^{-1/(n-1)}$ & 0.618 & 0.737& 0.804& 0.845 &0.872& 0.891& 0.906& 0.9170&  0.926\\
$t/(t+1)$ & 0.667 &0.600& 0.566& 0.546& 0.534 &0.525& 0.519& 0.515& 0.511\\\hline
\end{tabular}
\caption{Value of $\beta$, $|\lambda_2|$, $\beta^{-1/(n-1)}$, and $t/(1+t)$ for $n=2,3,\ldots,10$.}\label{tab:beta}
\end{table}

To further investigate the convergence rates we computed $\|f_m - f_\beta\|_\infty$ for $f(x)=\alpha x^{\alpha - 1}$, corresponding to a beta-distribution with shape parameters $\alpha\ge1 $ and $1$ ($f$ is Lipschitz if and only if $\alpha\ge1$). Figure~\ref{fig:1} shows the log error $\ln\|f_m-f_\beta\|_\infty$ when $n=2,3$ and 
$\alpha=1,2,3,4$. In each case the log error is approximately a linear function of $m$ and the corresponding regression line is shown in Figure~\ref{fig:1}. When $\alpha=2,3,4$ the slopes  of the regression lines divided by $-\ln(\beta)$ are $0.993,1.014,1.024$ for $n=2$ and $0.991,0.997,0.998$ for $n=3$. This suggests that the rate $\beta^{-mt/(1+t)}$ is not optimal and that the actual rate might be $\beta^{-m}$. The situation for $\alpha=1$ is quite different with a much faster decay, where the slope divided by $-\ln(\beta)$ is $2.000$ for $n=2$ and $1.485$ for $n=3$. Finally, since $f$ is left skewed when $\alpha>1$, we also considered results for the right skewed PDF $f(x)=2(1-x)$, however, the results are almost indistinguishable from those when $f(x)=2x$ (the case $\alpha=2$) and therefore plots similar to those in Figure~\ref{fig:1} are omitted. 


\begin{figure}
\includegraphics[width=\textwidth]{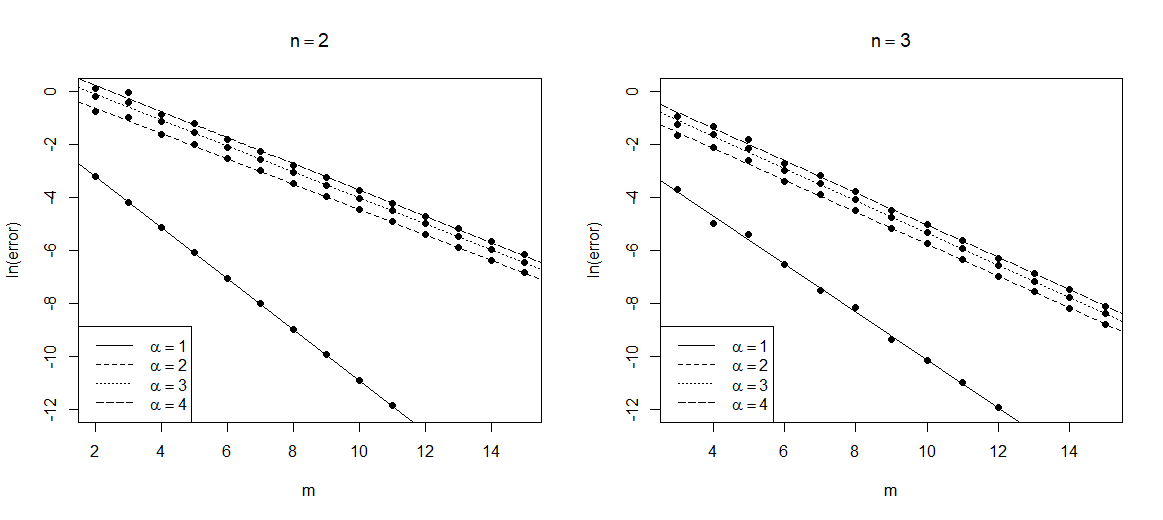}
\caption{Plot of the logarithm of the error $\|f_m - f_\beta\|_\infty$ as a function of $m$ for $\alpha=1,2,3,4$ and $n=2,3$. }\label{fig:1}
\end{figure}

\section{Formula for the PDF of $t^m(X)$}\label{sec:tm}

As we saw above  
$$t^m(X) = \sum_{k=1}^\infty X_{m+k}\beta^{-k}.$$
To determine the distribution of $t^m(X)$, we divide $[0,1)$ into intervals.
For  $m\in\mathbb N$ define 
$$\Omega_m = \{(j_1,j_2,\cdots,j_m)\,|\, j_k \in \{0,1\} \mbox{ and no more than $n-1$  $j_k$'s in succession equal 1}\}.$$
If $J=(j_1,\ldots,j_m)\in \Omega_m$ we define
$$L_{J,m} =  \sum_{k=1}^m j_k\beta^{-k}.$$
Let  $r(J)$ be the number of 1's in succession at the end of $J$. 
Then $0\le r(J) \le \min(m,n-1).$ Suppose $x = \sum_{k=1}^\infty \beta^{-k} x_k$ with $x_k = j_k$ for $k=1,...,m$ and the $x_k$'s following the standard restrictions.  Clearly $ L_{J,m}\le x$. We show in Appendix~A 
that 
\begin{align*}
	&x<  L_{J,m}+\beta^{-(m+1)} +\beta^{-(m+2)} +\cdots +\beta^{-(m+ n-r(J))}.
\end{align*}
Setting $$K(J) = \beta^{-1} + \cdots + \beta^{-( n-r(J))},$$ 
this leads us to define the interval 
\begin{equation}\label{IJdef} I_J = I_{J,m}= [L_{J,m}, L_{J,m} + \beta^{-m}K(J)).\end{equation}

\begin{proposition} \label{prop:IJ} The following properties are satisfied.
	\begin{enumerate}
		\item[1)] If $J,J'\in \Omega_m$ then $I_J\cap I_{J'} = \emptyset$ unless $J=J'$.
		\item[2)] $\cup_{J \in \Omega_m} I_J = [0,1)$.
		\item[3)] Any $x\in [0,1)$ has a unique expansion $x = \sum_{k=1}^\infty x_k\beta^{-k}$ with the $x_k$'s obeying the standard restrictions. 
		\item[4)] For $J\in \Omega_m$, $I_J$ consists of those  $x = \sum_{k=1}^\infty x_k\beta^{-k}$ with the $x_k$'s obeying the standard restrictions and $J=(x_1,\ldots,x_m)$.
	\end{enumerate}
\end{proposition}
\begin{proof}
	See  Appendix~A. 
\end{proof}

Recalling \eqref{e:main}
we have $P(t^m(X)\le x) = \sum_{J\in \Omega_m} P(t^m(X) \le x, X\in I_{J,m})$,  where  $t^m(X) \le x$ is the same as $\sum_{k=m+1}^\infty X_k\beta^{-k} \le \beta^{-m}x$, so 
\begin{equation} \label{PtmX}
	P(t^m(X) \le x, X\in I_{J,m}) = P(L_{J,m}\le X \le  L_{J,m} + \beta^{-m}x, X\in I_{J,m}).
\end{equation}
From \eqref{IJdef} and \eqref{PtmX} we have for every $x\in[0,1)$ that
\begin{align*}
	P(t^m(X) \le x) =  \sum_{J\in \Omega_m}[&\chi_{[0,K(J))}(x) P(L_{J,m} \le X \le L_{J,m} +\beta^{-m} x)\\
	&+\chi_{[K(J),1]}(x)P(L_{J,m} \le X \le L_{J,m} +\beta^{-m} K(J))].
\end{align*}
It follows that $t^m(X)$ has a PDF given by
\begin{equation} \label{fm}
	T^mf(x)=f_m(x) = \beta^{-m} \sum_{J\in \Omega_m} \chi_{[0,K(J))}(x) f(L_{J,m} + \beta^{-m}x).
\end{equation}

\section{On the cardinality of a certain set 
}\label{sec:Omega}

To proceed with the proof of Theorem \ref{mainthm}, we need to know how many intervals $J\in \Omega_m$ assume a given value of $r(J)$. For $m\in\mathbb N$ and $r\in\{0,...,\min(m,n-1)\}$ let $$\Omega_{m,r}= \{J\in \Omega_m\,|\, r(J) = r\}.$$
Denote the cardinality of $\Omega_{m,r}$ by $N_r(m)$. This is determined in Lemma~\ref{$N_s$} below using  the following considerations.

We have
$$N_0(m+1) =  \sum_{r=0}^ {\min(m,n-1)} N_r(m)$$  
and for $1\le r \le \min(m+1,n-1)$
$$N_r(m+1) = N_{r-1}(m).$$ For $1\le m\le n-1$ we have $N_0(m) = 2^{m-1}$ since $j_m=0$ and $j_k$ can be arbitrary for  $k<m$,  and for $2\le m\le n- 1$, $N_1(m) = 2^{m-2}$ since $j_m=1$, $j_{m-1} = 0$ and the remaining $j_k$ are arbitrary.  Similarly if $  r+1 \le m \le n-1$, $N_r(m) = 2^{m-(r+1)}$.  Finally $N_m(m) = 1$. 

Define  a matrix $M$ such that the equation 
\begin{equation} \label{Mdef}
	\begin{bmatrix}
		N_0(m+1) \\
		\vdots \\
		N_{n-1}(m+1)\\
	\end{bmatrix}
	=
	\begin{bmatrix}
		1&1&1& \cdots 1&1\\
		1&0&0& \cdots 0&0\\
		0&1&0& \cdots 0&0\\
		\vdots \\
		0&0&0& \cdots1&0\\
	\end{bmatrix}
	\times
	\begin{bmatrix}
		N_0(m) \\
		\vdots \\
		N_{n-1}(m)\\
	\end{bmatrix}
	=
	M \times 
	\begin{bmatrix}
		N_0(m) \\
		\vdots \\
		N_{n-1}(m)\\
	\end{bmatrix}.
\end{equation}   
holds whenever $m\ge n-1$. Thus
\begin{equation}  \label{findingN_s}
	\begin{bmatrix}
		N_0(m+1) \\
		\vdots \\
		N_{n-1}(m+1)\\
	\end{bmatrix}
	=
	M^{m+1-(n-1)} \times 
	\begin{bmatrix}
		N_0(n-1) \\
		\vdots \\
		N_{n-1}(n-1)\\
	\end{bmatrix} 
	= M^{m+1-(n-1)} \times 
	\begin{bmatrix}
		2^{n-2}\\
		\vdots \\
		2^{n-2-r}\\
		\vdots\\
		1\\
	\end{bmatrix}.
\end{equation}
It is not hard to show using expansion by minors of the first column of $xI -M$ that 
\begin{equation}\label{eq:eigen}
	\det ( xI -M) = (x-1)x^{n-1} - (x^{n-2} + x^{n-3} + \cdots + x + 1)  = x^n - (1 + x + \cdots + x^{n-1}).
\end{equation}
Thus $M$ has a unique positive eigenvalue $\beta$. Using \eqref{betan}, an unnormalized eigenvector corresponding to $\beta$ is seen to be 
\begin{equation}
	u =
	\begin{bmatrix}
		\beta^{n-1}\\
		\beta^{n-2}\\
		\vdots \\
		1\\
	\end{bmatrix}
\end{equation}
and an unnormalized eigenvector with eigenvalue $\beta$ for the transposed matrix $M^*$ is 
\begin{equation} \label{v}
	v= 
	\begin{bmatrix}
		1\\
		\beta-1 \\
		\beta^2 - \beta -1\\
		\vdots \\
		\beta^{n-1} -(\beta^{n-2} + \cdots + 1)
	\end{bmatrix}.
\end{equation}

The components of \eqref{v} all have the same sign.  In fact, for $j=1,...,n$,
\begin{equation} \label{identity}
	\beta^{j-1} - (\beta^{j-2} + 
	\beta^{j-3} + \cdots + 1) = \beta^{-1} + \beta^{-2} + \cdots + \beta^{-(n-(j-1))}.   
\end{equation}
To see this, subtract the right hand side from the left, factor out $\beta^{j-1}$ and use \eqref{betan}.

Let $\langle v,u\rangle $ be the usual inner product and define 
\begin{equation} \label{D}
	D(\beta) = \beta^{-(n-1)} \langle v,u\rangle  
\end{equation}
where $D(\beta)$ is the same constant as in  \eqref{fbeta1}, which can be verified by direct computation. Moreover, we calculate
$$ \langle v,u\rangle = n\beta^{n-1} -  \sum_{m=0}^{n-2} (m+1) \beta^m =  \frac{\mathrm d}{\mathrm dx}(x^n -x^{n-1} -x^{n-2} - \cdots - 1)\bigg|_{x=\beta} = \prod_{j=2}^{n} (\beta - \lambda_j)$$
where the $\lambda_j$'s are the roots of $x^n - (x^{n-1} + \cdots + 1)$.  We order the $\lambda_j$'s so that $\beta = \lambda_1 > |\lambda_2| \ge \cdots \ge |\lambda_{n}| $.  By \eqref{eq:eigen}  these are  the eigenvalues of $M$.

Let 
\begin{equation}
	w = 
	\begin{bmatrix}
		2^{n-2}\\
		\vdots \\
		2^{n-2-r}\\
		\vdots\\
		2^{n-n}\\
		1\\
	\end{bmatrix}.
\end{equation}

\begin{lemma} \label{vw} We have 
	\begin{equation}
		\langle v,w\rangle = \beta^{n-1}.  
	\end{equation}  
\end{lemma}

\begin{proof}
	The lemma is true for $\beta $ replaced by any $x\in \R$  not necessarily satisfying \eqref{betan}.  Thus, for the purpose of this proof, we replace $\beta$ by $x$ in \eqref{v}. First note that the lemma  is true for $n=2$.  Assuming it is true for $n$, we have for $n+1$
	\begin{align*} 
		\langle v,w\rangle =&\, 2[2^{n-2} + (x-1)2^{n-3} +\cdots + (x^{n-2} - (x^{n-3} + \cdots + 1))]  +  x^{n-1} - \\
		&\,(x^{n-2} + \cdots + 1) + x^n - (x^{n-1} + \cdots + 1)\\
		=&\, 2[x^{n-1} - (x^{n-1} - (x^{n-2} + \cdots + 1))] + x^{n-1} -\\ 
		&\,(x^{n-2} + \cdots + 1) + x^n - (x^{n-1} + \cdots + 1)\\
		=&\, x^n.
	\end{align*}
\end{proof}

We can now determine $N_r(m)$ from \eqref{findingN_s}.

\begin{lemma} \label{$N_s$}  We have for  $m\in\{n,n+1,...\}$ and $r\in\{0,...,\min(m,n-1)\}$ that
	$$N_r(m) = \beta^{m +n-1-r} \langle v,w\rangle ^{-1} + O(|\lambda_2|^{m}).$$
\end{lemma}

\begin{proof}
	We use the fact that the eigenvalues of $M$ are simple, see \cite{MI}.  Thus $M$ has a basis of eigenvectors, $Mu_j = \lambda_j u_j, 1\le j\le n, u_1 =u$, and similarly $M^*$ has a basis of eigenvectors, $M^*v_j = \bar {\lambda_j}v_j, 1\le j\le n, v_1 = v$.  
We have $\langle v_j,u_k\rangle= 0$ if $j\ne k$, since $(\lambda_k - \lambda_j)\langle v_j,u_k\rangle = \langle v_j,Mu_k\rangle - \langle M^*v_j,u_k\rangle = 0$. Thus 
if $w=\sum_{j=1}^n \alpha_ju_j$ we have $\langle v_j,w\rangle = \alpha_j \langle v_j,u_j\rangle$ or 
$$w= \sum_{j=1}^n \langle v_j,w\rangle\,\langle v_j,u_j\rangle^{-1}u_j.$$		
	It follows that  
	\[N_r(m)= \sum_{j=1}^{n} (M^{m-(n-1)} \langle v_j,w\rangle\,\langle v_j,u_j\rangle^{-1} u_j)_r 
	=\beta^{m-(n-1)}\langle v,w\rangle\,\langle v,u\rangle^{-1}(u)_r + O(|\lambda_2|^{m}).\]
	Using $\langle v,w\rangle = \beta^{n-1}$ and $ (u)_r = \beta^{n-1 -r}$, we find the result.
\end{proof}

\section{Proof of main theorem}\label{sec:proof}

We need the following alternative expression for $f_\beta$.

\begin{lemma}
	With $f_\beta$ as in \eqref{fbeta1},  we have for $0\le x<1$
	\begin{equation} \label{fbeta}
		f_\beta(x) = \sum_{s=0}^{n-1}\chi_{[0,\beta^{-1} + \cdots + \beta^{-(n-s)})}(x)\beta^{-s}D(\beta)^{-1} .   
	\end{equation}   
\end{lemma}

\begin{proof}
	We have 
	\[\sum_{s=0}^{n-1}\chi_{[0,\beta^{-1} + \cdots + \beta^{-(n-s)})}\beta^{-s}
	=\sum_{j=1}^{n}(1+ \beta^{-1} + \cdots + \beta^{-(n-j)})\chi_{[\beta^{-1} + \cdots + \beta^{-(j-1)}, \beta^{-1} + \cdots + \beta^{-j})}\]
	where the term for $j=1$ is interpreted as $(1+ \beta^{-1} + \cdots + \beta^{-(n-1)})\chi_{[0,\beta^{-1})}$.
	From \eqref{identity}, $$1+ \beta^{-1} + \cdots + \beta^{{-(n-j)}} = \beta^j-(\beta^{j-1} + \cdots + \beta),$$ which gives the result.
\end{proof}

In case $f$ is not Lipschitz we will use an approximation $g$ which is a Lipschitz continuous PDF on $[0,1)$ with $|g(x) - g(y)| \le L_g |x-y|$.   For $0\le x<1$ and $m\in\mathbb N$ we define
\begin{equation}\label{defgm}
	g_m(x) = \beta^{-m} \sum_{J\in \Omega_m} \chi_{[0,K(J))}(x) g(L_{J,m} + \beta^{-m}x).
\end{equation}
The proposition below is later used to show \eqref{fm-fbeta} in Theorem \ref{mainthm}.

\begin{proposition} \label{gm-fbetaest}
	With $f_\beta $ and $g_m$ defined in \eqref{fbeta1} and \eqref{defgm}, respectively,  if $m>2n$ then
	\begin{equation} \label{gm-fbeta}
		\|g_m - f_\beta\|_\infty = O(\beta^{-m_1})L_g + O((\beta^{-1}|\lambda_2|)^{m_2})
	\end{equation}
	whenever $m_1,m_2 \in \mathbb{N}$ with $m_2 > 2n-1$, $m_1\ge n-1$ and $m_1 +m_2 = m$, and where the $O$-terms do not depend on $g$.
	
\end{proposition}

\begin{proof}
	Fix $m_1+m_2 = m$.  Define, for fixed $J'\in \Omega_{m_1}$ with $m> m_1$, the set  $\Omega_m(J') = \{(j_1,j_2,..,j_m) \in \Omega_m \mid (j_1,j_2,..,j_{m_1}) =J'\}$. The idea is to localize to subintervals $I_{J}$ of the intervals $I_{J'} $ with $ J' \in \Omega_{m_1}$ and $J\in \Omega_m(J')$. 
	Note that $$\sum_{{J'}\in \Omega_m}h(J') = \sum_{J'\in \Omega_{m_1}}\sum_{J\in \Omega_m(J')}h(J).$$
	Thus from \eqref{defgm}
	\begin{equation}
		g_m(x) = \sum_{J'\in \Omega_{m_1}}\sum_{J\in \Omega_m(J')}\chi_{[0,K(J))}(x)\beta^{-m} g(L_{J,m} + \beta^{-m}x).
	\end{equation}
	We have $I_{J,m} \subset I_{J',m_1}$ and $I_{J',m_1} = \cup_{J\in \Omega_m(J')} I_{J,m}$ by 4) in Proposition \ref{prop:IJ}. For $y = L_{J,m} + \beta^{-m}x \in I_{J,m}$ and $t\in I_{J',m_1}$ we have $|y-t| \le |I_{J',m_1}| = K(J')\beta^{-m_1}$.   
	The cardinality of the set $\Omega_m$ is $N_0(m+1)$ which is given by Lemma~\ref{$N_s$}. Thus the cardinality of $\Omega_m$ is $O(\beta^m)$ so that 
	$$g_m(x) = \left[\sum_{J'\in \Omega_{m_1}}\sum_{J\in \Omega_m(J')}\bigg(\chi_{[0,K(J))}(x)\beta^{-m}|I_{J',m_1}|^{-1}\int_{I_{J',m_1}}g(t)\,\mathrm dt\bigg)\right] + O(\beta^{-m_1})L_g$$
	where $O(\beta^{-m_1})$ does not depend on $g$.
	Since $|I_{J',m_1}| = \beta^{-m_1} K(J')$,
	\begin{equation} \label{gm1}
		g_m(x) = \left[\sum_{J'\in \Omega_{m_1}}\sum_{J\in \Omega_m(J')}\bigg(\chi_{[0,K(J))}(x) K(J')^{-1}\beta^{-m+m_1}\int_{I_{J',m_1}}g(t)\,\mathrm dt\bigg)\right] + O(\beta^{-m_1})L_g .
	\end{equation}
	Let $\Omega_{m_1,r_1}$ be the set of $J'\in \Omega_{m_1}$ with $r(J') = r_1$. Each term in \eqref{gm1} depends only on $J$ through $K(J)$ which in turn depends only on $r(J)$. Thus, counting the number of $J$ with $r(J) = s$ and  $J\in \Omega_{m_1}(J')$ for some $J'$ with $r(J') = r_1$, we get 
	\begin{align*} \label{gm}
		& g_m(x) = \bigg[\sum_{r_1=0}^{n-1}\sum_{J'\in \Omega_{m_1,r_1}} \sum_{s=0}^{n-1}\chi_{[0,\beta^{-1} + \cdots + \beta^{-(n-s)})}(x)K(J')^{-1} \times  \\
		&( N_s(m_2-1) + N_s(m_2-2) + \cdots + N_s(m_2 - (n-r_1))) \beta^{-m_2} \int_{I_{J',m_1}} g(t)\,\mathrm dt\bigg ] + O(\beta^{-m_1})L_g 
	\end{align*}
	where for $k=1,...,n-r_1$, $N_s(m_2-k)$ is the number of $J\in\Omega_{m_1}(J')$ starting with $J'$ followed by $k-1$ 1's and ending with $s$ 1's. 
	From Lemma \ref{$N_s$} and since $m>2n$, for $s=0,...,n-1$ and $k=1,...,n-r_1$,
	$$N_s(m_2-k) =  \beta^{m_2-k +n-1-s} \langle v,u\rangle^{-1} + O(|\lambda_2|^{m_2-k}).$$   Thus 
	\begin{align*}&\sum_{s=0}^{n-1}\chi_{[0,\beta^{-1} + \cdots + \beta^{-(n-s)})}(x)K(J')^{-1} \times \\ &\hspace{5mm}(N_s(m_2-1) + N_s(m_2-2) + \cdots + N_s(m_2 - (n-r_1))) \beta^{-m_2} \\
		= \,& K(J')^{-1} \sum_{s=0}^{n-1}\sum_{k=1}^{n-r_1} \beta^{n-1 -s -k}\chi_{[0,\beta^{-1} + \cdots + \beta^{-(n-s)})}(x)\langle v,u\rangle^{-1} + O(|\beta^{-1}\lambda_2|^{m_2}).
	\end{align*}
	We have $\sum_{k=1}^{n-r_1}\beta^{-k} = K(J')$.  Thus 
	\begin{align*}
		&g_m(x) = \\
		& \sum_{J'\in \Omega_{m_1}}\sum_{s=0}^{n-1}\chi_{[0,\beta^{-1} + \cdots + \beta^{-(n-s)})}(x)\beta^{n-1-s} \int_{I_{J',m_1}} g(t)\,\mathrm dt\,\langle v,u\rangle^{-1} + O(|\beta^{-1}\lambda_2|^{m_2}) + O(\beta^{-m_1})L_g\\
		&= \sum_{s=0}^{n-1}\chi_{[0,\beta^{-1} + \cdots + \beta^{-(n-s)})}(x)\beta^{n-1-s}\langle v,u\rangle^{-1} + O(|\beta^{-1}\lambda_2|^{m_2}) + O(\beta^{-m_1})L_g
	\end{align*} 
	where $O(|\beta^{-1}\lambda_2|^{m_2})$ does not depend on $g$.
	Using \eqref{D} and \eqref{fbeta} we have 
	$$g_m(x) = f_\beta(x)  + O(|\beta^{-1}\lambda_2|^{m_2}) + O(\beta^{-m_1})L_g.$$
	This is \eqref{gm-fbeta}. 
\end{proof}

For the proof of the remaining part of Theorem \ref{mainthm}, we will use that 
\begin{equation} \label{Tys}
	d_{TV}(P_m,P_\beta) = \frac{1}{2}\int_{[0,1)} |f_m(x)-f_\beta(x)|  \,\mathrm dx 
\end{equation}
(see \cite{Tysbakov}, Lemma 2.1).
With $T$ as defined in Section~\ref{s:1}, we have 
\begin{equation} \label{L1est}
	\|f_m -g_m\|_{L^1} = \|T^m f-T^m g\|_{L^1} \le \|f-g\|_{L^1}, 
\end{equation}
since $T$  takes PDFs to PDFs and thus has an extension to $L^1([0,1))$ satisfying $\|Th\|_{L^1} \le \|h\|_{L^1}$ for all $h\in L^1([0,1))$.

\begin{proof}[Proof of Theorem \ref{mainthm}]
	Using \eqref{gm-fbeta} and \eqref{L1est} we have $$\|f_m-f_\beta\|_{L^1} \le \|f_m - g_m\|_{L^1} + \|g_m - f_\beta\|_{L^1} \le\|f - g\|_{L^1} + O(\beta^{-m/2})(1+ L_g).$$ 
	Thus $$\limsup_{m\to \infty} \|f_m-f_\beta\|_{L^1} \le \|f-g\|_{L^1},$$ which can be made as small as we like. Using \eqref{Tys}, this gives
	$$d_{TV}(P_m,P_\beta) \rightarrow 0.$$
	Equation \eqref{fm-fbeta} follows from \eqref{gm-fbeta} setting $f = g$ and $f_m = g_m$.
\end{proof}

\section*{Appendix A: The beta expansion}

In this appendix we prove Proposition \ref{prop:IJ}. We use the notation from the main text. We let $S=(1,\ldots,1,0,1,\ldots,1,0,\ldots)$ denote the sequence having alternatingly $n-1$ 1's in succession and a 0.

\begin{lemma}\label{lem:unique}
	For any sequence of $x_k$'s satisfying the standard restrictions, $\sum_{k=1}^\infty x_k\beta^{-k}$ is a number in $[0,1)$.
	If for $x\in [0,1)$ we let $x_k = \lfloor \beta t^{k-1}(x)\rfloor$, then $x=\sum_{k=1}^\infty x_k\beta^{-k} $ is the unique $\beta$-expansion of $x$ satisfying the standard restrictions. 
\end{lemma}

\begin{proof}
	To show the first statement, note that the largest value of a $\beta$-expansion without $n$ 1's in succession is $\sum_{i=1}^{n-1} \sum_{l=0}^\infty \beta^{-i-nl} = \sum_{l=0}^\infty\beta^{-nl}(1-\beta^{-n}) =1$  corresponding to the sequence $(x_1,x_2,\ldots)=S$. The $\beta$-expansion of any other sequence of $x_k$'s without $n$ 1's in succession yields a strictly smaller number (indeed, for any such sequence one obtains a strictly larger $\beta$-expansion by swapping $(0,1)$ to $(1,0)$ repeatedly moving all 1's as far to the left as possible and (if there are finitely many 1's) adding extra 1's until $S$ is reached).
	
	Setting $x_k = \lfloor \beta t^{k-1}(x)\rfloor$ we have  $|x-\sum_{k=1}^m x_k\beta^{-k}| = \beta^{-m}t^{m}(x)< \beta^{-m}   $ which goes to zero for $m\to \infty$ showing that  $x =\sum_{k=1}^\infty x_k\beta^{-k}$. This expansion must satisfy the standard restrictions. Indeed, if  $(x_{k+1},\ldots,x_{k+n})=(1,\ldots,1)$ for some $k\ge 0$, then $ t^{k}(x)\ge \sum_{i=1}^n x_{k+i} \beta^{-i} = 1$,  which is a contradiction. Similarly, if $(x_{k+1},\ldots)=S $, then $ t^{k}(x)= \sum_{i=1}^{n-1} \sum_{l=0}^\infty \beta^{-i-nl}= \sum_{l=0}^\infty \beta^{-nl}(1-\beta^{-n}) = 1$ which is again a contradiction.

	To show uniqueness, suppose $\sum_{k=1}^\infty x_k \beta^{-k} = \sum_{k=1}^\infty y_k \beta^{-k}$ and let $m$ be minimal with $x_m\neq y_m$, say $x_m=0$ and $y_m=1$. Then $\sum_{k=1}^\infty x_{m+k} \beta^{-k} =1+ \sum_{k=1}^\infty y_{k+m} \beta^{-k}$. By the above considerations, the only way this can hold is if $(x_{m+1},x_{m+2},\ldots) = S$ and $(y_{m+1},y_{m+2},\ldots) = (0,0,\ldots)$ showing that only the $y_k$'s satisfy the standard restrictions.
\end{proof}

\begin{lemma}\label{lem:lex}
	If $x,y\in [0,1)$ have $\beta$-expansions $x=\sum_{k=1}^\infty x_{k} \beta^{-k}$ and $y=\sum_{k=1}^\infty y_k \beta^{-k}$ following the standard restrictions, then $x<y$ if and only if the sequence $(x_1,x_2,\ldots)$ is lexicographically smaller than $(y_1,y_2,\ldots)$. 
\end{lemma}

\begin{proof}
	Let $m$ be minimal with $x_m\neq y_m$. Then $x<y$ if and only if
	$$t^{m-1}(x)=\sum_{k=1}^\infty x_{k+m-1} \beta^{-k} < \sum_{k=1}^\infty y_{k+m-1} \beta^{-k}=t^{m-1}(y).$$	
	Since $x_m\neq y_m$, this is equivalent to $x_m = \lfloor \beta t^{m-1}(x) \rfloor=0$ and $y_m=\lfloor \beta t^{m-1}(y) \rfloor=1$.
\end{proof}

\begin{lemma}\label{lem:disj}
	If $x\in [0,1)$ has $\beta$-expansion $x=\sum_{k=1}^\infty x_k \beta^{-k}$ following the standard restrictions and $J\in \Omega_m$, then $(x_1,\ldots,x_m)=J$ if and only if $x\in I_J$. 
\end{lemma}

\begin{proof}
	Any tail sequence $(x_{l+1},x_{l+2},\ldots)$ also follows the standard restrictions for any $l\in \mathbb{N}$ and hence $\sum_{k=l+1}^\infty x_k \beta^{-k} <\beta^{-l}$ by Lemma \ref{lem:unique}. Let $r(J)=r$ and let $s\le n-r-1$ be the number of 1's at the beginning of $(x_{m+1},x_{m+2},\ldots)$. Then $\sum_{k=m+1}^\infty x_k \beta^{-k} < \sum_{i=1}^s \beta^{-(m+i)} + \beta^{-(m + s+1 )} \le \sum_{i=1}^{n-r} \beta^{-(m+i)} $ and hence $L_{J,m}\le x < L_{J,m}+K(J)\beta^{-m}$. This shows $x\in I_J$. 
	
	On the other hand assume that $(x_1,\ldots,x_m) =K\neq J=(j_1,\ldots,j_m)$ and let $k$ be minimal with $x_k\neq j_k$. If  $x_k<  j_k$, then $x<L_{J,m}$ by Lemma \ref{lem:lex}. If $j_k<x_k$, then $y<x$ for any $y$ with $(y_1,\ldots,y_m)=J$. We can approximate the upper bound on $I_J$ arbitrarily well by $\beta$-expansions corresponding to finite truncations of the sequence $(j_1,\ldots,j_m,1,\ldots,1,0,1,\ldots,1,\ldots)$ where there are $n-r-1$ 1's in the first sequence of 1's and $n-1$ 1's in all following sequences. Lemma \ref{lem:lex} implies that $x\ge L_{J,m}+\beta^{-m}K(J)$.
\end{proof}

\begin{proof}[Proof of Proposition \ref{prop:IJ}]
	Since all $x\in [0,1)$ have a $\beta$-expansion $x=\sum_{k=1}^\infty x_k\beta^{-k}$ following the standard restrictions by Lemma \ref{lem:unique}, Lemma \ref{lem:disj} shows that $x\in I_J$ for $J=(x_1,\ldots,x_m)$. This shows 2). Since the $\beta$-expansion following the standard restrictions is unique by Lemma \ref{lem:unique}, Lemma \ref{lem:disj} shows 1). Finally 3)  follows from Lemma \ref{lem:unique} and 4) follows from Lemma \ref{lem:disj}. 
\end{proof}

		\section*{Appendix B: Some properties of the roots  of \eqref{eq:eigen}}
		
		\begin{lemma} \label{roots}
			Fix $n\ge 2$ and consider the polynomial $p(x)=x^n - (1 + \cdots + x^{(n-1)})$. 
			\begin{itemize} 
				\item[1)] There is a unique positive 
				root $\beta$ of $p$.  This satisfies $\beta \in (1,2)$.
				\item[2)] If $n$ is even there is one negative root.  All other roots are non-real. If $n$ is odd all roots other than $\beta$ are non-real.
				\item[3)] If we label the roots $\lambda_j$ with $|\lambda_1| > |\lambda_2| \ge \cdots \ge |\lambda_n|$, with $\lambda_1 = \beta$, then $|\lambda_j| < 1$ for $j\ge 2$ and $|\lambda_2| \ge \beta^{-1/(n-1)}$ for all $n$.  If $n$ is even then $\lambda_n < 0$ and in fact $|\lambda_n| < |\lambda_{n-1}|$.  We also have $|\lambda_j|\ge 3^{-1/n}$ for $j=1,...,n$. The inequality $|\lambda_2|\ge \beta^{-1/(n-1)}$ is actually an equality for $n=2$ and $3$. 
				\item[4)] All roots are simple.
				\item[5)] $\beta$ is increasing in $n$.  In fact $2-\beta = \beta^{-n}$ so that $\beta$ approaches $2$ exponentially fast as $n$ increases.
			\end{itemize} 
		\end{lemma}
		
		
		\begin{proof}
			Much of this is well known, see \cite{Meyer,MI}. 
			We note first that $p(x)$ is the determinant of the matrix $xI-M$,  {cf.\ \eqref{Mdef}}. 
			$M$ has non-negative entries and $M^n$ has all positive entries and therefore it is well known \cite{Meyer} that there is a unique positive root which is strictly larger than the absolute value of all other roots. All roots are simple and $|\lambda_j| < 1$ for $j >1$ (see \cite{MI}). 
			The equation $x^n = 1 + x + \cdots + x^{n-1}$ can be rewritten $1 = x^{-1} + x^{-2} +\cdots + x^{-n} $.  From this equation it is easily seen that the positive solution, $\beta > 1$, increases with $n$. We have $1+ x + \cdots + x^{n-1} = (1- x^n)/(1-x)$ so that we can rewrite the first equation as
			$$x^n = (x^n -1)/(x-1)$$ or 
			$$ x^{n+1} + 1 = 2x^n$$
			which gives $$2-x = x^{-n}.$$
			This shows $\beta < 2$ and since $\beta$ increases with $n$,  $2- \beta $ must decrease exponentially in $n$.
			If we factor  $p(x) = \prod_{j=1}^n (x-\lambda_j)$ we see that $\prod_{j=1}^n\lambda_j = (-1)^{n+1} $ which implies $\beta |\lambda_2|^{n-1} \ge 1$ or $|\lambda_2|\ge \beta^{-1/(n-1)}$.  From the equation $x^{n+1} + 1 = 2x^n$ we find $2|x|^n \ge 1-|x|^{n+1}$ or $|x|^n \ge (2+|x|)^{-1}$.  Applying this to $x = \lambda_j, j\ge 2$, using $|\lambda_j| <1$, we find $|\lambda_j| > 3^{-1/n}$.  In particular all roots other than $\beta$ approach the unit circle as $n$ increases.  
			
			Suppose $x$ is a negative root.  If $n$ is odd we have $-2|x|^n = 1+ |x|^{n+1}$ which is impossible, so  {there are no negative roots if $n$ is odd.  If $n$ is even we have $2|x|^n = 1-|x|^{n+1}$}.  If $\lambda$ is another root then $2|\lambda|^n= |1+\lambda^{n+1}| \ge 1-|\lambda|^{n+1}$.  Subtracting we find $2(|\lambda|^n - |x|^n) \ge |x|^{n+1} - |\lambda|^{n+1}$ and therefore $|\lambda| \ge |x|$.  If $|\lambda| = |x|$ then $2|x|^n = 1-|x|^{n+1}$ implies $2|\lambda|^n = 1-|\lambda|^{n+1} = |1 + \lambda^{n+1}|$. This implies $\lambda^{n+1} < 0$. Since $2\lambda^n = 1+ \lambda^{n+1}$ we must have $\lambda <0$.  Finally this means $\lambda = x$. Thus for $n$  even  there is at most one negative root  and all other roots are strictly larger in absolute value. Moreover, we  know there is  one positive root and as before $\prod_{j=1}^n \lambda_j = (-1)^{n+1}$.  Since all non-real roots come in complex conjugate pairs this shows that if $n$ is even there is precisely one negative root.
			
			If $n=2$ or $3$ we have $\beta \lambda_2 = -1$ and $\beta |\lambda_2|^2 = 1$ respectively,  so in both cases we have $|\lambda_2| = \beta^{-1/(n-1)}.$
		\end{proof}


%
%
%
%
%

\end{document}